\documentclass[11pt]{amsart}
\usepackage[active]{srcltx}
\usepackage{verbatim}
\usepackage{epsfig,graphicx,color,mathrsfs}
\usepackage{graphicx}
\usepackage{amsmath,amssymb,amsthm,amsfonts}
\usepackage{amssymb}
\usepackage[english]{babel}
\usepackage[scrtime]{prelim2e}
\usepackage{epsfig,graphicx,color,mathrsfs}
\usepackage[english]{babel}
\usepackage[left=2.7cm,right=2.7cm,top=3cm,bottom=3cm]{geometry}

\numberwithin{equation}{section}
\newtheorem{theorem}{Theorem}[section]
\newtheorem{proposition}[theorem]{Proposition}
\newtheorem{lemma}[theorem]{Lemma}

\newtheorem{remark}[theorem]{Remark}

\theoremstyle{definition}

\renewcommand{\dfrac}{\displaystyle\frac}
\newcommand{\brm}{\begin{remark}\rm}
\newcommand{\erm}{\end{remark}}
\newcommand{\brms}{\begin{remark}\rm}
\newcommand{\erms}{\end{remark}}
\newcommand{\bte}{\begin{theorem}}
\newcommand{\ete}{\end{theorem}}
\newcommand{\bpr}{\begin{proposition}}
\newcommand{\epr}{\end{proposition}}
\newcommand{\ble}{\begin{lemma}}
\newcommand{\ele}{\end{lemma}}
\newcommand{\beq}{\begin{equation}}
\newcommand{\eeq}{\end{equation}}
\newcommand{\bdm}{\begin{displaymath}}
\newcommand{\edm}{\end{displaymath}}
\numberwithin{equation}{section}

\newcommand{\bos}{\begin{remark}\rm}
\newcommand{\eos}{\end{remark}}

\newcommand{\ben}{\begin{enumerate}}
\newcommand{\een}{\end{enumerate}}

\newcommand{\be}{\begin{equation}}
\newcommand{\ee}{\end{equation}}

\title[The strong comparison principle]{On the Harnack inequality for quasilinear elliptic equations with a first order term}

\author[S.\ Merch\'an]{Susana Merch\'an$^*$}

\author[L.\ Montoro]{Luigi Montoro$^*$}
\author[B.\ Sciunzi]{Berardino Sciunzi$^*$}

\thanks{$^*$Dipartimento di Matematica e Informatica,
Universit\`a della Calabria,
Ponte Pietro Bucci 31B, I-87036 Arcavacata di Rende, Cosenza, Italy,
E-mail: {\em merchan@mat.unical.it}, {\em montoro@mat.unical.it}, {\em sciunzi@mat.unical.it}}

\begin{document}

\begin{abstract}
We consider weak solutions to $$-\Delta_pu+a(x,u)|\nabla u|^q=f(x,u),$$
with $p>1$, $q\geq\max\,\{p-1,1\}$.
We exploit the Moser iteration technique to prove a Harnack comparison inequality for $C^1$ weak solutions.  As a consequence we deduce a strong comparison principle.
\end{abstract}

\maketitle

%\tableofcontents

%%%%%%%%%%%%%%%%%%%%%%%%%%%%%%%%%%%%%%
\medskip
\section{Introduction}\label{introdu}
Let us consider positive weak   $C^1_{loc}({\Omega})$  solutions   to  the problem
\begin{equation}\label{p}
-\Delta_pu+a(x,u)|\nabla u|^q=f(x,u)\hspace{5pt}\hbox{in}\;\;\;\Omega,\\
\end{equation}
where $p>1$, $q\geq\max\,\{p-1,1\}$,  $\Omega$ is a  domain in $\mathbb{R}^N$ and  $N\geq 2$. 
It is well known that the  $C^{1,\alpha}_{loc}$-regularity of the solutions is natural when dealing with such problems (see \cite{Di,surveymin,ZIGMUND,ming2,ming3, tex, tex2,   T}).
The functions $a(x,s)$ and
$f(x,s)$ obey to the set of suitable assumption $(hp^*)$ detailed in Section \ref{pre}. Let us emphasize that we mainly need the source term $f(x,s)$
to be positive in order to apply our technique. \\

We exploit the Moser iteration technique to derive a Harnack comparison inequality. Actually the method that we use is one developed by Trudinger in
\cite{Tru} to study a degenerate class of operators in weighted Sobolev spaces. We deduce both the weak and the strong Harnack comparison inequality and we may resume our main result in the following

\begin{theorem}\label{Harna} (Harnack Comparison Inequality).
Let $p>(2N+2)/(N+2)$ and let $u,v\,\in\,C^1_{loc}(\Omega)$   with $u$ or $v$ weak solution to  \eqref{p} in $\Omega$. Let
$q\geq\max\,\{p-1,1\},$
and assume that $f(x,u), a(x,u)$ fulfill  $(hp^*)$. Suppose   that $\overline{B(x,6\delta)}\subset\Omega'\subset \Omega$ for some $\delta>0$ and that $$u\leq v \quad \text{in}\quad {B(x,6\delta)}.$$
Then there exists  $C=C(p,q,\delta,L, \|v\|_{L^{\infty}(\Omega')},  \|\nabla u\|_{L^{\infty}(\Omega')},  \|\nabla v\|_{L^{\infty}(\Omega')})>0$ such that
\begin{equation}\label{HarnIneq}
\sup_{B(x,\delta)}(v-u)\leq C\inf_{B(x,2\delta)}(v-u).
\end{equation}
\end{theorem}
Our problem is related to the study of Trudinger \cite{Tru} mainly because of the fact that  it can be studied in weighted Sobolev spaces and the natural weight is the weight $\rho=|\nabla u|^{p-2}$ (or $\rho=(|\nabla u|+|\nabla v|)^{p-2}$) which is degenerate ($p>2$) or singular ($p<2$) on the critical set
$$Z_u:=\{x\in\Omega\,\,|\,\,  \nabla u(x)=0\}.$$
In the singular case $1<p<2$ the condition $p>(2N+2)/(N+2)$ provides integrability properties of the weight (see \cite{DS}). It is worth emphasizing that the weight is not in $L^1$ in general if $p$ is close to one. The problem without first order terms have been studied in \cite{DinoLu, SciDi1} and  the same result in our case is somehow expected. Some related problems are studied in \cite{LPR}. Our main effort is to obtain such a Harnack type inequality under suitable general assumptions, having in mind possible applications in the study of qualitative properties of the solutions.
An important consequence of the Harnack comparison inequality is in fact the strong comparison principle for  \eqref{p} that we point out in the following
\begin{theorem}\label{thm:strongggggggg}(Strong Comparison Principle)
Let $p>(2N+2)/(N+2)$ and  $u,v\in C^1_{loc}(\Omega)$ with either $u$ or $v$  weak solution to~\eqref{p}. Assume that
$ q\geq\max\,\{p-1,1\}$ and assume that $f(x,u), a(x,u)$ fulfill $(hp^*)$. Then, if
\begin{equation}\nonumber
-\Delta_pu+a(x,u)|\nabla u|^q-f(x,u)\leq -\Delta_pv+a(x,v)|\nabla v|^q-f(x,v),\quad u\leq v \quad  \text{in}\quad {\Omega}
\end{equation}
in the weak distributional meaning, it follows that
$$u<v \quad \text{in} \quad \Omega$$
unless $u\equiv v$ in $\Omega$.
\end{theorem}
The strong comparison principle for $p$-Laplace equations is a very delicate issue and manly still unsolved. Actually it is not hard to derive it far from the critical set, see e.g. \cite{lucio, PSB} but it remains an open problem already for $p$-harmonic functions (see \cite{HKM}) near critical points. So our result is crucial in particular to work in regions where the gradient of the solutions vanishes. In particular we are motivated by the possible applications in many issues and in particular in the study of qualitative properties of the solutions. We refer in particular to the papers   \cite{FMRS,MMPS,MRS} where it is clear that the strong principle can simplify the proofs and improve the results.\\

\noindent The paper is organized as follows: In Section \ref{pre} we recall some preliminary results. The proof of the main result is a consequence of the results in Section \ref{Harnack} where we prove the weak and the strong Harnack Comparison Inequality.

\section{Preliminary results}\label{pre}
In this section we start recalling some useful regularity results about  solutions to problem~\eqref{p}.
Through all the paper, generic fixed and numerical constants will be denoted by
$C$ (with subscript or superscript in some case) and they will be allowed to vary within a single line or formula.
We assume that $a=a(x,u)$ and $f=f(x,u)$  satisfy the following hypotheses (denoted by $({hp^*})$ in the sequel):

\

\begin{itemize}
\item [$({hp^*})$]
\begin{itemize}
\item  $a(x,\cdot)\in C^1( \Omega' \times [0,+\infty))$ for any $\Omega'\subset\Omega$.
 \item $f(x,\cdot)$ is positive and, more precisely,  $f(x,s)>0$ in $\Omega'$ for every  $\Omega'\subset \Omega$ and for every $s>0$.
\item $a(x,\cdot)$ and $f(x,\cdot)$ are locally Lipschitz continuous, uniformly w.r.t. $x$. Namely, for every  $\Omega'\subset \Omega$ and for every $M>0$,  there is a  positive constant $L=(M,\Omega')$ such that for every $x \in \Omega'$ and every $ u,v \in [0,M] $ we have:
$$ \vert a(x,u) - a(x,v) \vert \le L \vert u-v \vert, \qquad \vert f(x,u) - f(x,v) \vert \le L \vert u-v \vert. $$
\end{itemize}
\end{itemize}
For the reader's convenience (and since in the sequel we will use the hypotheses $(hp^*)$ also  in this form),  we remark  that  the hypotheses $(hp^*)$ imply the following: for every  $\Omega'\subset \Omega$ and for every $M>0$, there exists  $K=K(M,\Omega')>0$, such that for every $x \in \Omega'$ and every $ s \in [0,M] $ we have
$$ \vert a(x,s) \vert \le K.$$
By standard regularity results, see  \cite{Di,Li,T}, the solutions to problems involving the $-\Delta_p(\cdot)$ operator, (and under suitable hypotheses) are in general of class $C^{1,\alpha}$. This fact leads to the study of the summability properties of the second derivatives of the solutions that turns out to be crucial in our results.
We recall a regularity result in \cite{DS, SciDi2}, see also \cite{MSS} for the case of  equations with first order terms.
\begin{theorem}\label{local1}
Let $1<p<\infty$ and consider $u\in C^{1,\alpha}_{loc}(\Omega)$ a weak solution to \eqref{p}, with $a(x,\cdot)$ and $f(x,\cdot)$ satisfying $(hp^*)$. Denoting $u_i=\frac{\partial u}{\partial x_i}$, for any $\Omega'\subset\Omega''\subset\Omega$, we have
\begin{equation}\label{drdrd}
\begin{split}
&\int_{\Omega'} \frac{|\nabla u|^{p-2-\beta}|\nabla u_i|^2}{|x-y|^\gamma}\,dx\leqslant \mathcal{C}\qquad \forall \, i=1,\ldots,N,\\
\end{split}
\end{equation}
uniformly for any $y\in\Omega'$,
with
\[
\mathcal{C}:\,=
\mathcal{C}\Big(a, f,p, q,\beta, \gamma, \|u\|_{L^\infty(\Omega'')},\|\nabla u\|_{L^\infty(\Omega'')}\Big),
\]
for any $0\leqslant \beta <1$ and $\gamma <(N-2)$ if $N\geq3$, or $\gamma =0$ if $N=2$.\\

Moreover, if $f(x,\cdot)$ is positive in $\Omega''$, then it follows that
\begin{equation}\label{drdrdbisssetebissete}
\begin{split}
&\int_{\Omega'} \frac{1}{|\nabla u|^{r(p-1)}}\frac{1}{|x-y|^\gamma}dx\leqslant {\mathcal{C}}^*,\\
\end{split}
\end{equation}
uniformly for any $y\in\Omega'$,
with
\[
\mathcal{C^*}:\,=
\mathcal{C^*}\Big(a, f,p, q,r, \gamma, \|u\|_{L^\infty(\Omega'')},\|\nabla u\|_{L^\infty(\Omega'')}\Big),
\]
for any  $ r<1$ and $\gamma <(N-2)$ if $N\geq3$, or $\gamma =0$ if $N=2$.
\end{theorem}
We refer to \cite{DS, MSS} for a detailed proof. Note that, by \eqref{drdrdbisssetebissete}, it follows in particular that
the critical set  $Z_u$ has zero Lebesgue measure provided that
 $f(x,\cdot)$ is positive.\\

\noindent We now recall that, for $\rho \in L^{1}(\Omega)$ and $1\leq s<\infty$,  the space $H^{1,s}_\rho(\Omega)$ is
defined as the completion of $C^1(\Omega)$ (or $C^{\infty }(\Omega)$) with the norm
  \begin{equation}\label{hthInorI}
\| v\|_{H^{1,s}_\rho}= \| v\|_{L^s (\Omega)}+\| \nabla v\|_{L^s (\Omega, \rho)},
\end{equation}
 where
$$
\|\nabla v\|^s_{L^s (\Omega, \rho)}:=\int_{\Omega}\rho(x)|\nabla v(x)|^s  dx.
$$
We also have that $H^{1,s}_{\rho}(\Omega)$ may be equivalently defined as the space of
functions with distributional derivatives represented by a function for which the norm defined in
\eqref{hthInorI} is bounded. The space $H^{1,s}_{0,\rho}(\Omega)$ is consequently defined as the closure of $C^1_c(\Omega)$ (or $C^{\infty }_c(\Omega)$), w.r.t. the norm \eqref{hthInorI}.  We refer to \cite{DS} for more details about weighted Sobolev spaces and also to \cite[Chapter 1]{HKM} and the references therein.

\

Theorem \ref{local1} provides the right summability of the weight $|\nabla u(x)|^{p-2}$ in order to obtain a weighted Poincar\'e-Sobolev  type inequality that will be useful in the sequel. For the proof we refer to \cite[Section 3]{DS}.
\begin{theorem}[Weighted Poincar\'e-Sobolev type inequality]\label{bvbdvvbidvldjbvlb}
Let $p\geq2$ and let $u\in C^{1,\alpha}_{loc}(\Omega)$ be a solution to~\eqref{p}  with $a(x,\cdot)$ and $f(x,\cdot)$ satisfying $(hp^*)$. For any $\Omega'\subset\Omega$, setting $\rho=|\nabla u|^{p-2}$, we have  that  $H^{1,2}_{0}(\Omega',\rho)$ is continuously embedded in $L^{q}(\Omega')$ for  $1\leq q< \bar 2_p$, where
\[
\frac{1}{\bar{ 2}_p}=\frac{1}{2}-\frac{1}{N}+\frac{p-2}{p-1}\,\frac{1}{N}\,.
\]
Consequently, since $\bar 2_p>2$, for $w\in H^{1,2}_{0}(\Omega',\rho)$ we have\begin{equation}\label{Sobolev}
\|w\|_{L^2(\Omega')}\leqslant \mathcal{C}_p\|\nabla w\|_{L^2(\Omega', \rho)}=\mathcal{C}_p\left(\int_{\Omega'}\rho\, |\nabla w|^2\right)^{\frac{1}{2}},
\end{equation}
with $\mathcal{C}_p=\mathcal{C}_p(\Omega')  \rightarrow 0$ if $|\Omega'|\rightarrow 0$.
\end{theorem}
Theorem \ref{bvbdvvbidvldjbvlb} follows from Theorem \ref{local1}, see \cite{DS}. The proof it is based on potential estimates, see \cite[Lemmas 7.14, 7.16]{GiTru}. Since potential estimates are also available for functions with zero mean, we can prove Theorem  \ref{Sobolev} for functions with zero mean (we will refer to it below in the paper), in particular see \cite[Theorem 8, Corollary 2]{FMS3}.
Moreover in the paper we will use the fact that, since
 $$0\leq \frac{1}{(|\nabla u|+|\nabla v|)^{r(p-1)}}\frac{1}{|x-y|^\gamma}\leq\frac{1}{|\nabla u|^{r(p-1)}}\frac{1}{|x-y|^\gamma},$$from \eqref{drdrdbisssetebissete} it readily follows
\begin{equation}\nonumber
\int_{\Omega} \frac{1}{(|\nabla u|+|\nabla v|)^{r(p-1)}}\frac{1}{|x-y|^\gamma}dx\leq C^*.
\end{equation}
Therefore, in particular, thanks to Theorem \ref{bvbdvvbidvldjbvlb}  we also  have   a weighted Poincar\'e-Sobolev inequality with weight $\rho=(|\nabla u|+|\nabla v|)^{p-2}$, for any $v\in C^1_{loc}(\Omega)$.
Finally  notice that Theorem \ref{bvbdvvbidvldjbvlb} holds for $p\geq 2$. Anyway, if $1<p<2$ and $|\nabla u|$ is bounded,  the weighted Poincar\'{e}-Sobolev
inequality \eqref{Sobolev} follows at once by the classic Poincar\'{e}-Sobolev inequality.

In the sequel we use the following standard estimates, whose proof can be found e.g. in \cite{lucio}.
 \begin{lemma}
$ \forall \, p>1 $, there exist positive constants $\hat C, \check C$,  depending on $p$, such that  $\; \forall \, \eta, \eta' \in  \mathbb{R}^{N}$ with $|\eta|+|\eta'|>0$
\begin{eqnarray}\label{eq:inequalities}
[|\eta|^{p-2}\eta-|\eta'|^{p-2}\eta'][\eta- \eta'] &\geq& \hat C (|\eta|+|\eta'|)^{p-2}|\eta-\eta'|^2\\ \nonumber
\\\nonumber
&\text{and}&
\\\nonumber
\\\nonumber
||\eta|^{p-2}\eta-|\eta'|^{p-2}\eta'| &\leq& \check C(|\eta|+|\eta'|)^{p-2}|\eta -\eta'|.\end{eqnarray}
 \end{lemma}
\section{The Harnack Comparison Inequality}\label{Harnack}
In this section we  show the steps needed to prove a strong Harnack inequality.   We start with the following results:
\begin{theorem}(Weak Harnack Comparison Inequality)\label{t1}
Let $u,v\in C^1_{loc}(\Omega)$ and assume that either $u$ or $v$ is a weak solution to \eqref{p}, with $ q\geq\max\,\{p-1,1\}$
and $f(x,u), a(x,u)$ satisfying  $(hp^*)$.  Assume  that $\overline{B(x,6\delta)}\subset\Omega'\subset\Omega$ for some $\delta>0$ and that
\begin{equation}\label{eq:ekinotizia}
-\Delta_pu+a(x,u)|\nabla u|^q-f(x,u)\leq -\Delta_pv+a(x,v)|\nabla v|^q-f(x,v),\quad u\leq v \quad  \text{in}\quad {B(x,6\delta)}.
\end{equation}
We distinguish the two cases:

\

\begin{itemize}
\item {
Case $(a):p\geq 2$.} Define
\[
\frac{1}{\bar{ 2}_p}=\frac{1}{2}-\frac{1}{N}+\frac{p-2}{p-1}\,\frac{1}{N}.
\]
Then, for every
\begin{equation}\nonumber
0<s<\dfrac{\bar 2_p}{2},
\end{equation}
there exists $C>0$ such that
\begin{equation}\label{tt1}
\|(v-u)\|_{L^s(B(x,2\delta))}\leq C\inf_{B(x,\delta)}(v-u)
\end{equation}
where $C=C(p,q,\delta,L, \|v\|_{L^{\infty}(\Omega')},  \|\nabla u\|_{L^{\infty}(\Omega')},  \|\nabla v\|_{L^{\infty}(\Omega')})$.

\

\item  {Case $(b):{(2N+2)}/{(N+2)}<p<2$.} Define
$$\frac{1}{\bar t^\sharp}=\frac{2p-3}{2(p-1)}$$
and let   $2^*$ be  the classical Sobolev exponent $2^*=2N/(N-2)$.
Then, for every
\begin{equation}\nonumber
0<s<\dfrac{2^*}{\bar t^\sharp},
\end{equation}
there exists $C>0$ such that
\begin{equation}\label{tt1}
\|(v-u)\|_{L^s(B(x,2\delta))}\leq C\inf_{B(x,\delta)}(v-u)
\end{equation}
where $C=C(p,q,\delta,L, \|v\|_{L^{\infty}(\Omega')},  \|\nabla u\|_{L^{\infty}(\Omega')},  \|\nabla v\|_{L^{\infty}(\Omega')})$.
\end{itemize}
\end{theorem}
The proof is based on the  Moser iteration scheme, see \cite{MO}. Actually we exploit here the improved technique due to
 Trudinger \cite{Tru} which is only based on weighted Sobolev inequalities and avoid the use of John-Nirenberg's Lemma.
\begin{proof}
Let us first note that, since we are assuming that $u,v\in C^1_{loc}(\Omega)$, we need to work in a sub-domain $\Omega'$. To simplify the  notation we relabel it as $\Omega$ in all the proof.
Let us consider the function $w_{\tau}=v-u+\tau$, where $\tau >0$ in $\Omega$ (at the end we will let $\tau\rightarrow 0$).   Moreover let us define the function $\phi_{\tau}=\eta^2w_{\tau}^{\beta}$ with $\beta<0$ and $\eta\in C_0^1(B(x,6\delta))$. Then  it follows
\begin{equation}\label{derifi}
\nabla\phi_{\tau}=2\eta w_{\tau}^{\beta}\nabla\eta+\beta\eta^2 w_{\tau}^{\beta -1}\nabla w_{\tau}.
\end{equation}
Using $\phi_{\tau}$ positive as a test function in  \eqref{eq:ekinotizia}, it follows
\begin{eqnarray}\label{eq:sgahjgsgh}
&&\int_{\Omega}(|\nabla u|^{p-2}\nabla u-|\nabla v|^{p-2}\nabla v, \nabla\phi_{\tau})dx\\\nonumber
&\leq& \int_{\Omega}(a(x,v)|\nabla v|^q-a(x,u)|\nabla u|^q)\phi_{\tau}dx+\int_{\Omega}(f(x,u)-f(x,v))\phi_{\tau}dx.
\end{eqnarray}

\

\noindent {Case $(a)$: $p\geq 2$.} Taking into account \eqref{derifi}, the first term in \eqref{eq:sgahjgsgh} can be written as
\begin{eqnarray}\nonumber
&&\int_{\Omega}(|\nabla u|^{p-2}\nabla u-|\nabla v|^{p-2}\nabla v,\nabla\phi_{\tau})dx\\ \nonumber
\\ \nonumber
&=&\int_{\Omega}(|\nabla u|^{p-2}\nabla u-|\nabla v|^{p-2}\nabla v, \nabla w_{\tau})\beta\eta^2w_{\tau}^{\beta -1}dx\\ \nonumber
\\ \nonumber
&+&2\int_{\Omega}(|\nabla u|^{p-2}\nabla u-|\nabla v|^{p-2}\nabla v, \nabla\eta)\eta w_{\tau}^{\beta}dx.
\end{eqnarray}
Using the  inequalities \eqref{eq:inequalities}, from \eqref{eq:sgahjgsgh} it follows
\begin{eqnarray}\label{mec}
&&\hat C|\beta|\int_{\Omega}(|\nabla u|+|\nabla v|)^{p-2}|\nabla (u-v)|^2\eta^2w_{\tau}^{\beta-1}dx\\ \nonumber
&\leq &\int_{\Omega}(a(x,v)|\nabla v|^q-a(x,u)|\nabla u|^q)\phi_{\tau}dx+\int_{\Omega}(f(x,u)-f(x,v))\phi_{\tau}dx\\ \nonumber
&+&2\int_{\Omega}(|\nabla u|^{p-2}\nabla u-|\nabla v|^{p-2}\nabla v,\nabla\eta)\eta w_{\tau}^{\beta}dx.
\end{eqnarray}
Applying the second inequality of  \eqref{eq:inequalities} and  Young's inequality   in the last term of the previous expression, we obtain
\begin{eqnarray}\label{eq:dgashkhkaw}
&&2\left |\int_{\Omega}(|\nabla u|^{p-2}\nabla u-|\nabla v|^{p-2}\nabla v, \nabla\eta)\eta w_{\tau}^{\beta}dx\right |\\ \nonumber
&\leq &2\check C\int_{\Omega}(|\nabla u|+|\nabla v|)^{p-2}|\nabla (u-v)||\nabla\eta|\eta w_{\tau}^\beta dx
\\ \nonumber
&\leq &2\check C\int_{\Omega}(|\nabla u|+|\nabla v|)^{\frac{p-2}{2}}(|\nabla u|+|\nabla v|)^{\frac{p-2}{2}}|\nabla (u-v)||\nabla\eta|\eta w_{\tau}^{\frac{\beta -1}{2}}w_{\tau}^{\frac{\beta +1}{2}}dx
\\\nonumber
&\leq &2 \check C\varepsilon\int_{\Omega}(|\nabla u|+|\nabla v|)^{p-2}|\nabla (u-v)|^2\eta ^2 w_{\tau}^{\beta -1}+\frac{2 \check C}{\varepsilon}\int_{\Omega}(|\nabla u|+|\nabla v|)^{p-2}|\nabla\eta|^2w_{\tau}^{\beta +1}dx.
\end{eqnarray}
Therefore using  \eqref{eq:dgashkhkaw}, setting $\rho=(|\nabla u|+|\nabla v|)^{p-2}$ and choosing $\varepsilon={\hat C|\beta|}/{4\check C}$,  \eqref{eq:dgashkhkaw} becomes
\begin{eqnarray}\label{eq:ooooooooo}
&&\frac{\hat C |\beta|}{2}\int_{\Omega}\rho|\nabla (u-v)|^2\eta^2w_{\tau}^{\beta -1}dx\\ \nonumber
&\leq &\int_{\Omega}(a(x,v)|\nabla v|^q-a(x,u)|\nabla u|^q)\phi_{\tau}dx+\int_{\Omega}(f(x,u)-f(x,v))\phi_{\tau}dx\\ \nonumber
&+&\frac{8\check C^2}{\hat C|\beta|}\int_{\Omega}\rho|\nabla\eta|^2w_{\tau}^{\beta +1}dx.
\end{eqnarray}
We estimate now the first two terms on the right-hand side of \eqref{eq:ooooooooo}. In particular, using hypotheses $(hp^*)$ and the mean value theorem,   for the first term we have that
\begin{eqnarray}\label{eq:fahkjslghksl}
&&\int_{\Omega}(a(x,v)|\nabla v|^q-a(x,u)|\nabla u|^q)\phi_{\tau}dx=\int_{\Omega}(a(x,v)|\nabla v|^q-a(x,u)|\nabla u|^q)\eta ^2 w_{\tau}^{\beta}dx
\\\nonumber
&&=\int_{\Omega}(a(x,v)|\nabla v|^q-a(x,v)|\nabla u|^q)\phi_{\tau}dx+\int_{\Omega}(a(x,v)-a(x,u))|\nabla u|^q\phi_{\tau}dx\\\nonumber
&&\leq K(\|v\|_{L^{\infty}(\Omega)})\int_{\Omega}|(|\nabla v|^q-|\nabla u|^q)|\eta ^2w_{\tau}^{\beta}dx+\int_{\Omega}|(a(x,v)-a(x,u))||\nabla u|^q\phi_{\tau}dx
\\\nonumber&&\leq q\int_{\Omega}(|\nabla u|+|\nabla v|)^{q-1}|\nabla(v-u)|\eta ^2w_{\tau}^{\beta}dx+ L(\|v\|_{L^{\infty}(\Omega)})\cdot\|\nabla u\|^q_{L^{\infty}(\Omega)}\int_{\Omega}(v-u)\phi_{\tau}dx
\\\nonumber
&&\leq q\int_{\Omega}(|\nabla u|+|\nabla v|)^{\frac{p-2}{2}}|\nabla(v-u)|\eta w_{\tau}^{\frac{\beta -1}{2}}\eta w_{\tau}^{\frac{\beta +1}{2}}\dfrac{(|\nabla u|+|\nabla v|)^{q-1}}{(|\nabla u|+|\nabla v|)^{\frac{p-2}{2}}}dx\\\nonumber
&&+ L(\|v\|_{L^{\infty}(\Omega)})\cdot\|\nabla u\|^q_{L^{\infty}(\Omega)}\int_{\Omega}(v-u+\tau)\phi_{\tau}dx\\\nonumber
&& \leq\varepsilon q\int_{\Omega}\rho|\nabla (v-u)|^2\eta ^2w_{\tau}^{\beta -1}dx+\frac{C(p,q,\|\nabla u\|_{L^{\infty}(\Omega)},  \|\nabla v\|_{L^{\infty}(\Omega)})}{\varepsilon}\int_{\Omega}\eta^2w_{\tau}^{\beta +1}dx\\\nonumber
&&+C(L,\|v\|_{L^{\infty}(\Omega)},\|\nabla u\|_{L^{\infty}(\Omega)})\int_{\Omega}\eta^2 w_{\tau}^{\beta +1}dx,
\end{eqnarray}
where in the last line  we applied the Young's inequality and we used the fact  that $q\geq p/2.$
For the second term on the right-hand side of \eqref{eq:ooooooooo}, using $(hp^*)$ we obtain that
\begin{eqnarray}\label{eq:fafaffasfafgsagfasg}
&&\int_{\Omega}(f(x,u)-f(x,v))\phi_{\tau}dx\\\nonumber
&&\leq \int_{\Omega}|(f(x,u)-f(x,v))|\phi_{\tau}dx\leq L(\|v\|_{L^{\infty}(\Omega)})\int_{\Omega}(v-u+\tau)\phi_{\tau}dx
\\\nonumber&& \leq L(\|v\|_{L^{\infty}(\Omega)})\int_{\Omega}\eta^2 w_{\tau}^{\beta +1}dx.
\end{eqnarray}
Choosing $\varepsilon =\frac{\hat C|\beta|}{4q}$ and using  \eqref{eq:fahkjslghksl} and \eqref{eq:fafaffasfafgsagfasg}, from \eqref{eq:ooooooooo} we have
\begin{eqnarray}\nonumber
&&\frac{\hat C |\beta|}{4} \int_{\Omega}\rho|\nabla (v-u)|^2\eta ^2w_{\tau}^{\beta -1}dx\\ \nonumber
&&\leq\frac{C(p,q,\|\nabla u\|_{L^{\infty}(\Omega)},  \|\nabla v\|_{L^{\infty}(\Omega)})}{|\beta|}\int_{\Omega}\eta^2w_{\tau}^{\beta +1}dx\\\nonumber
&&+C(L,\|v\|_{L^{\infty}(\Omega)},\|\nabla u\|_{L^{\infty}(\Omega)})\int_{\Omega}\eta^2 w_{\tau}^{\beta +1}dx +
 \frac{C(p)}{|\beta|}\int_{\Omega}\rho|\nabla\eta|^2w_{\tau}^{\beta +1}dx.
\end{eqnarray}
Therefore, up to  redefining constants, we obtain
\begin{eqnarray}\label{luigipanza}
\int_{\Omega}\rho|\nabla w_{\tau}|^2\eta ^2w_{\tau}^{\beta -1}dx
\leq \dfrac{\dot C}{|\beta|}\left (1+\dfrac{1}{|\beta|}\right)\int_{\Omega}w_{\tau}^{\beta +1}(\eta ^2 +\rho|\nabla\eta|^2)dx,
\end{eqnarray}
where $\dot C=\dot C(p,q,L, \|v\|_{L^{\infty}(\Omega)},  \|\nabla u\|_{L^{\infty}(\Omega)},  \|\nabla v\|_{L^{\infty}(\Omega)})$ is a positive constant.
\

Let us now set
\begin{equation}\label{eq:cazzz}
\tilde w_{\tau}=\begin{cases}
w_{\tau}^{\frac{\beta +1}{2}} & \hbox{if}\;\;\;\beta\neq -1,\\
\log(w_{\tau}) &\hbox{if}\;\;\;\beta =-1
\end{cases}
\end{equation}
and
\begin{equation}\label{eq:rrrrrr}
r:=\beta +1.
\end{equation}
Then, if $\beta\neq -1$, by  \eqref{eq:cazzz} it follows

\begin{equation}\int_{\Omega}\rho\eta^2|\nabla \tilde w_{\tau}|^2dx=\int_{\Omega}\rho\eta^2\left |\nabla(w_{\tau})^{\frac{\beta +1}{2}}\right |^2=\left (\dfrac{\beta +1}{2}\right )^2\int_{\Omega}\rho\eta^2 w_{\tau}^{\beta -1}|\nabla w_{\tau}|^2dx\end{equation}
and taking into account \eqref{luigipanza}, it follows
\begin{equation}\label{beta}
\int_{\Omega}\rho\eta^2|\nabla \tilde w_{\tau}|^2dx\leq \dot C \left (\dfrac{\beta +1}{2}\right )^2\dfrac{1}{|\beta|}\left (1+\dfrac{1}{|\beta|}\right )\int_{\Omega}\tilde w_{\tau}^{2}(\eta^2+\rho|\nabla\eta|)^2dx.
\end{equation}
Note that, since
\begin{equation}\nonumber
\frac{1}{(|\nabla u|+|\nabla v|)^{p-2}}\leq\frac{1}{|\nabla u|^{p-2}}\quad\text{and}\quad \frac{1}{(|\nabla u|+|\nabla v|)^{p-2}}\leq\frac{1}{|\nabla v|^{p-2}},
\end{equation}
then the weight $\rho$ satisfies the same properties of $|\nabla u|^{p-2}$ and $|\nabla v|^{p-2}$. Therefore since either $u$ or $v$ is a solution to \eqref{p}, a weighted Sobolev inequality is available in this case (see Theorem~\ref{bvbdvvbidvldjbvlb}).
Hence, since we can assume that $2_p>2$, let $2<\nu< \bar 2_p$. Using Theorem \ref{bvbdvvbidvldjbvlb} (since $\eta\tilde{w}_{\tau}\in H^{1,2}_{0}(\Omega,\rho)$),  it follows
\begin{eqnarray}\nonumber
\|\eta\tilde{w}_{\tau}\|^2_{L^{\nu}(\Omega)}&\leq &C_p\int_{\Omega} \rho|\nabla\eta\tilde{w}_{\tau}+\eta\nabla \tilde{w}_{\tau}|^2dx\\\nonumber
&\leq& 2 C_p\int_{\Omega}\rho\tilde{w}_{\tau}^2|\nabla\eta|^2+\rho\eta^2|\nabla\tilde{w}_{\tau}|^2dx.
\end{eqnarray}
Using now \eqref{beta} (together with \eqref{eq:cazzz} and \eqref{eq:rrrrrr}), it follows
\begin{eqnarray}\label{norma}
\|\eta\tilde{w}_{\tau}\|^2_{L^{\nu}(\Omega)}&\leq&  \dot C  \dfrac{r^2}{|\beta|}\left (1+\dfrac{1}{|\beta|}\right )\int_{\Omega} \tilde{w}^2_{\tau}(\eta^2+\rho|\nabla\eta|^2)dx\\\nonumber
&\leq&  \dot C \dfrac{r^2}{|\beta|}\left (1+\dfrac{1}{|\beta|}\right )\|\tilde{w}_{\tau}(\eta+|\nabla\eta|)\|^2_{L^2(\Omega)},
\end{eqnarray}
up to redefine the constant $\dot C$.
Moreover we note  that the quantity
\begin{equation}\nonumber
\dfrac{1}{|\beta|}\left (1+\dfrac{1}{|\beta|}\right)
\end{equation}
is bounded if $|\beta|\geq C>0$  and  then, from now on, it will be included in the constant $\dot C$.

Consider now $\delta\leq h'\leq h''\leq 5\delta$ and let us suppose  $\eta\equiv 1$ in $B(x,h')$ and $\eta\equiv 0$ outside $B(x,h'')$ with  $|\nabla\eta|\leq\dfrac{2}{h''-h'}$. Taking into account these assumptions  \eqref{norma} becomes
\begin{eqnarray}\label{eq:ghhgghhghgghhg}
\|\tilde{w}_{\tau}\|_{L^{\nu}(B(x,h'))}&\leq& \dot C|r|\|\tilde{w}_{\tau}(\eta+|\nabla\eta|)\|_{L^2(B(x,h''))}
\\\nonumber&\leq& \dot C\dfrac{|r|}{h''-h'}\|\tilde{w}_{\tau}\|_{L^2(B(x,h''))}.
\end{eqnarray}
Set $\chi={\nu}/{2}$ and notice that $\chi>1$. Considering $0<r<1$ (i.e. $-1<\beta\leq C<0$), it follows
\begin{eqnarray}\label{eq:robbennn}
\left(\int_{B(x,h')}(w_{\tau})^{{\frac{\nu(\beta +1)}{2}}}dx\right)^{\frac{1}{\chi\cdot r}}&=&\left(\int_{B(x,h')}w_{\tau}^{\chi\cdot r}dx\right)^{\frac{1}{\chi \cdot r }}
\\\nonumber
&\leq&\dfrac{(\dot C|r|)^{\frac 2r}}{(h''-h')^{\frac 2r}}\left(\int_{B(x,h'')}w_{\tau}^rdx\right)^{\frac 1r}.
\end{eqnarray}
Defining now the functional
\begin{equation}\label{eq:funzionalecolnumero}
\phi(p,r,v)=\left(\int_{B(x,r)}|v|^pdx\right)^\frac 1p.
\end{equation}
Then for $0<r<1$ (and in general for $r>0$ for future use), it follows from \eqref{eq:robbennn} that
\begin{equation}\label{14}
\phi(\chi r,h',w_{\tau})\leq \dfrac{(\dot C|r|)^{\frac 2r}}{(h''-h')^{\frac 2r}} \phi(r,h'',w_{\tau}).
\end{equation}
If instead $r<0$ (i.e. $\beta <-1$), using \eqref{eq:ghhgghhghgghhg}  and arguing as above,  we have
\begin{equation}\label{azzzzzc}
\phi(\chi r,h',w_{\tau})\geq \dfrac{(\dot C|r|)^{\frac 2r}}{(h''-h')^{\frac 2r}} \phi(r,h'',w_{\tau}).
\end{equation}
We exploit  now the  Moser's iterative technique, see \cite{MO}. For $r_0>0$ given, we define
\begin{equation}\label{eq:r000000}
r_k=(-r_0)\chi^k\quad \text{and} \quad h_k=\delta\left (1+\dfrac 32\left (\dfrac 12\right)^k\right ).
\end{equation}
It follows that $r_k\to -\infty$ and $\beta_k:=r_k-1\to -\infty$. Moreover
\begin{equation}\nonumber
 h_0=\dfrac{5\delta}{2}\quad \text {and}\quad h_k\to \delta\quad \text{as}\,\, k\rightarrow +\infty
 \end{equation}
 and
\begin{equation}\nonumber
 h_k-h_{k+1}=\dfrac{3}{2}\dfrac{\delta}{2^{k+1}}.
\end{equation}
Using these definitions, we iterate the expression of $\phi(\cdot, \cdot,\cdot)$ in \eqref{azzzzzc} obtaining
\begin{eqnarray}\label{eq:reiterating}
\phi(r_{k+1},h_{k+1},w_{\tau})&=&\phi(\chi r_k,h_{k+1},w_{\tau})\geq\dfrac{\left(\dot C|r_k|\right)^{\frac{2}{r_k}}}{\left (h_k-h_{k+1}\right )^{\frac{2}{r_k}}}\phi(r_k,h_k,w_{\tau})\\\nonumber&=&\dfrac{\left (\dot C|r_0|\chi^k\right)^{\frac{2}{(-r_0)\chi^k}}}{\left (\dfrac 32\dfrac{\delta}{2^{k+1}}\right)^{\frac{2}{(-r_0)\chi^k}}}\phi(r_k,h_k,w_{\tau})\\\nonumber&=&\left (\dot C^\frac{2}{-r_0}\right)^{\frac{1}{\chi^k}}\left [(|r_0|\chi^k)^{\frac{2}{-r_0}}\right]^{\frac{1}{\chi^k}}\left[\left (\dfrac 32\dfrac{\delta}{2^{k+1}}\right)^{\frac{2}{r_0}}\right]^{\frac{1}{\chi^k}}\phi(r_k,h_k,w_{\tau}).
\end{eqnarray}
Reiterating \eqref{eq:reiterating} and collecting  terms we get
\begin{equation}\label{eq:caldaczz}
\phi(r_{k+1},h_{k+1},w_{\tau})\geq \dot C^{\sum_{k\geq 0}\frac{1}{\chi^k}}\left(2\chi\right)^{\frac{2}{-r_0}\cdot\sum_{k\geq 0}\frac{k}{\chi^k}}\delta^{\frac{2}{r_0}\cdot\sum_{0\leq j\leq k}\frac{1}{\chi^j}}\phi(-r_0,\dfrac{5\delta}{2},w_{\tau}).
\end{equation}
We remark that in \eqref{eq:caldaczz} we have redefined the constant $\dot C$.
Since, by definition $\chi >1$, the series converge and it is possible to find a positive constant
$C=C(p,q,L, \|v\|_{L^{\infty}(\Omega)},  \|\nabla u\|_{L^{\infty}(\Omega)},  \|\nabla v\|_{L^{\infty}(\Omega)})$ such that
\begin{equation}\label{17}
\phi(-\infty,\delta,w_{\tau})\geq C\phi(-r_0,\dfrac{5\delta}{2},w_{\tau}).
\end{equation}
We are going to suppose now (we will prove it later) that  there exists $r_0>0$ and a constant $C>0$ such that
\begin{equation}\label{asterisco}
\phi(r_0,\dfrac{5\delta}{2},w_{\tau})\leq C\phi(-r_0,\dfrac{5\delta}{2},w_{\tau}).
\end{equation}

For $0<s\leq r_0$, using H\"older inequality, we have that
\begin{eqnarray}\label{eq:assasasasasassasa}
\phi(s,\dfrac{5\delta}{2},w_{\tau})=\left(\int_{B(x,\frac{5\delta}{2})}(w_{\tau})^{s}\right)^{\frac 1s}&\leq& \left(\int_{B(x,\frac{5\delta}{2})}(w_{\tau})^{r_0}\right)^{\frac{1}{r_0}}\left|B(x,\frac{5\delta}{2})\right|^{\frac{r_0-s}{r_0}} \\\nonumber
&\leq&C(r_0,s,\delta)\phi(r_0,\dfrac{5\delta}{2},w_{\tau}).
\end{eqnarray}
Using \eqref{17}, \eqref{asterisco}, \eqref{eq:assasasasasassasa}  and the fact that $\phi(s,\dfrac{5\delta}{2},w_{\tau})\geq \phi(s,2\delta,w_{\tau})$, we get
\begin{equation}\label{eqhjkdskjhjksd}
\phi(-\infty,\delta,w_{\tau})\geq C\phi(s,2\delta,w_{\tau}),
\end{equation}
with $C=C(p,q,L,\Omega, \|v\|_{L^{\infty}(\Omega)},  \|\nabla u\|_{L^{\infty}(\Omega)},  \|\nabla v\|_{L^{\infty}(\Omega)})$ a positive constant.
Therefore, taking the limit for $\tau$ that goes to zero in \eqref{eqhjkdskjhjksd}, since
\begin{equation}\label{eq:gilbardddd}
\phi(-\infty,\delta,w_{\tau})=\inf_{B(x,\delta)}w_{\tau},\end{equation}
then  \eqref{tt1} follows for $0<s\leq r_0$.

If instead, $r_0<s<\chi$, we take a finite number of iterations in \eqref{14}. In this case  we define
\begin{equation}\label{eq:anannananaan}
r_1=\dfrac{s}{\chi^{k_0+1}}\leq r_0,
\end{equation}
for a natural number $k_0$ large enough.

We want to point out that, in order  to apply \eqref{14}, we need to choose $r<1$ in such formula. For example in the first iteration of \eqref{14} (see also equation \eqref{21} below) we set   $r_{k_0+1}=(r_1\chi^{k_0})\chi$. Then we need to impose $r_1\chi^{k_0}<1$. This holds (as it can be proved using the definition \eqref{eq:anannananaan}) for $r_0<s<\chi$. Such condition is also sufficient for the other steps.

Hence we consider, for  $k=0,...,k_0+1$, the values
$$r_k=r_1\chi^k$$ and
$$h_0=\dfrac{5\delta}{2}>h_1>...>h_{k_0+1}=2\delta.$$ Using these assumptions, we can iterate \eqref{14} $k_0$ times, obtaining
\begin{equation}\label{21}
\phi(s,2\delta,w_{\tau})\leq C\phi(r_1,\dfrac{5\delta}{2},w_{\tau}).
\end{equation}
Since $0<r_1\leq r_0$, we claim that \eqref{asterisco} holds for $r_0$ replaced by $r_1$. In fact
we readily have
\begin{equation}\label{noncredo}
\phi(r_1,\dfrac{5\delta}{2},w_{\tau})\leq \phi(r_0,\dfrac{5\delta}{2},w_{\tau})\leq C\phi(-r_0,\dfrac{5\delta}{2},w_{\tau})\,.
\end{equation}
Moreover, using   H\"older  inequality, we also obtain
\begin{equation}\nonumber
\left(\int_{B(x,\frac{5\delta}{2})}\left (\frac{1}{w_{\tau}}\right)^{^{r_1}}\right)^{\frac{1}{r_1}}\leq C(r_0,r_1,\delta) \left(\int_{B(x,\frac{5\delta}{2})}\left (\frac{1}{w_{\tau}}\right)^{^{r_0}}\right)^{\frac{1}{r_0}},
\end{equation}
namely
\begin{equation}\nonumber
\phi(-r_1,\dfrac{5\delta}{2},w_{\tau})=\left(\int_{B(x,\frac{5\delta}{2})}\left (\frac{1}{w_{\tau}}\right)^{^{r_1}}\right)^{-\frac{1}{r_1}}\geq C(r_0,r_1,\delta) \left(\int_{B(x,\frac{5\delta}{2})}\left (\frac{1}{w_{\tau}}\right)^{^{r_0}}\right)^{-\frac{1}{r_0}}.
\end{equation}
Including the last expression in \eqref{noncredo}, it follows
\begin{equation}\nonumber
\phi(r_1,\dfrac{5\delta}{2},w_{\tau})\leq \phi(r_0,\dfrac{5\delta}{2},w_{\tau})\leq C\phi(-r_0,\dfrac{5\delta}{2},w_{\tau})\leq C\phi(-r_1,\dfrac{5\delta}{2},w_{\tau}),
\end{equation}
that is our claim.
Therefore, taking into account also \eqref{21}, we obtain
\begin{equation}\label{eq:jsajahajhajsjhioss}
\phi(s,2\delta,w_{\tau})\leq  C\phi(-r_1,\dfrac{5\delta}{2},w_{\tau}).
\end{equation}
Arguing exactly as above, also \eqref{17} can be obtained with $r_0$ replaced by $r_1$ getting
\begin{equation}\label{17bis}
\phi(-\infty,\delta,w_{\tau})\geq C\phi(-r_1,\dfrac{5\delta}{2},w_{\tau}).
\end{equation}
Finally, from \eqref{eq:jsajahajhajsjhioss} and \eqref{17bis},  using \eqref{eq:gilbardddd}, we have
\eqref{tt1}  for $r_0<s<\chi$.

\

\noindent To conclude the  proof of the theorem, we must show that there exists  $r_0>0$ and a positive constant $C$,  for which \eqref{asterisco} holds. We follow closely the technique introduced  in \cite{Tru}.

We define
\begin{equation}\nonumber
\tilde w_{\tau} =log (w_{\tau}),
\end{equation}
that is $\tilde w_{\tau}$ in \eqref{eq:cazzz} for $\beta=-1$. In this case using \eqref{luigipanza} we have
\begin{equation}\label{eq:bonamassaaaaa}
\int_{\Omega}\rho|\nabla \tilde w_{\tau}|^2\eta^2dx=\int_{\Omega}\rho\frac{|\nabla w_\tau|^2}{w^2_\tau}\eta^2dx\leq \dot C\int_{\Omega}(\eta ^2 +\rho|\nabla\eta|^2)dx,
\end{equation}
for some positive constant $\dot C=\dot C(p,q,L, \|v\|_{L^{\infty}(\Omega)},  \|\nabla u\|_{L^{\infty}(\Omega)},  \|\nabla v\|_{L^{\infty}(\Omega)})$.
>From \eqref{eq:bonamassaaaaa}, with $\eta=1$ in $B(x,5\delta)$, we obtain
\begin{equation}\label{constante}
\int_{B(x,5\delta)}\rho|\nabla\tilde {w}_\tau|^2dx\leq C,
\end{equation}
with $C$ not depending on $\tilde w_\tau$.
Replacing (in \eqref{eq:bonamassaaaaa}) $w_\tau$ by $w_\tau/k$, with \begin{equation}\nonumber
k:=e^{\frac{1}{|B(x,5\delta)|}\int_{B(x,5\delta)}\log \tilde w_\tau dx},
\end{equation}
we can suppose that $\tilde w_\tau$ has zero mean in $B(x,5\delta)$. Then we exploit Theorem \ref{bvbdvvbidvldjbvlb} and by  \eqref{constante} it follows that
\begin{equation}\label{eq:anni900000}
\|\tilde{w}_\tau\|_{L^{\nu}(B(x,5\delta))}\leq C,
\end{equation}
where $C$ is a constant not depending on $\tilde{w}_\tau$. This will be crucial when we will pass to the limit for $\tau \rightarrow 0$. The constant $k$ that we introduced does not modify the following calculations and can be cancelled in the conclusive inequality, i.e.  \eqref{eq:jsakwqDEWkabkB} below.

We are going to use (in the weak expression of \eqref{eq:ekinotizia}) the test function $$\phi=\eta^2\dfrac{1}{w_{\tau}}(|\tilde{w}_\tau|^{\beta}+(2\beta)^{\beta})\qquad \beta\geq 1,$$
 with $\eta\geq 0$ and $\eta\in C_0^1(B(x,5\delta))$. Therefore
\begin{equation}\label{derivadafi}
\nabla\phi=\dfrac{2\eta}{w_{\tau}}(|\tilde{w}_\tau|^{\beta}+(2\beta)^{\beta})\nabla\eta +\eta^2\dfrac{1}{w_{\tau}^2}(\beta sign(\tilde{w}_\tau)|\tilde{w}_\tau|^{\beta -1}-|\tilde{w}_\tau|^{\beta}-(2\beta)^{\beta})\nabla w_{\tau}
\end{equation}
and
\begin{eqnarray}\nonumber
&&\frac{1}{k}\int_{\Omega}(|\nabla v|^{p-2}\nabla v-|\nabla u|^{p-2}\nabla u,\nabla\phi) dx+\frac{1}{k}\int_{\Omega}(a(x,v)|\nabla v|^q-a(x,u)|\nabla u|^q)\phi dx\\\nonumber
&&\geq \frac{1}{k} \int_{\Omega}(f(x,v)-f(x,u))\phi dx.
\end{eqnarray}
Hence, using \eqref{derivadafi} and the definition of $\phi$, we get
\begin{eqnarray}\label{caos}
&&\frac{1}{k} \int_{\Omega}(|\nabla v|^{p-2}\nabla v-|\nabla u|^{p-2}\nabla u,\nabla\eta)\dfrac{2\eta}{w_{\tau}}(|\tilde{w}_\tau|^{\beta}+(2\beta)^{\beta})  dx
\\\nonumber
&+&\frac{1}{k}\int_{\Omega}(|\nabla v|^{p-2}\nabla v-|\nabla u|^{p-2}\nabla u,\nabla w_{\tau})\eta^2\dfrac{1}{w_{\tau}^2}(\beta sign(\tilde {w}_\tau)|\tilde{w}_\tau|^{\beta -1}-|\tilde{w}_\tau|^{\beta}-(2\beta)^{\beta})dx\\\nonumber
&+&\frac{1}{k} \int_{\Omega}(a(x,v)|\nabla v|^q-a(x,u)|\nabla u|^q)\eta^2\dfrac{1}{w_{\tau}}(|\tilde{w}_\tau|^{\beta}+(2\beta)^{\beta}) dx\\ \nonumber
&\geq&\frac{1}{k} \int_{\Omega}(f(x,v)-f(x,u))\eta^2\dfrac{1}{w_{\tau}}(|\tilde{w}|^{\beta}+(2\beta)^{\beta}) dx.\end{eqnarray}
We  estimate and rearrange  the  terms in the inequality \eqref{caos}. We start from
\begin{eqnarray}\label{eq:dhaslfhgwrilaioi}
&&\frac{1}{k}\int_{\Omega}(|\nabla v|^{p-2}\nabla v-|\nabla u|^{p-2}\nabla u,\nabla\eta)\dfrac{2\eta}{w_{\tau}}(|\tilde{w}_\tau|^{\beta}+(2\beta)^{\beta})  dx
\\\nonumber
&\leq&\check C(p) \int_{\Omega}\rho\dfrac{2\eta}{w_{\tau}}(|\tilde{w}_\tau|^{\beta}+(2\beta)^{\beta})|\nabla w_\tau||\nabla\eta|  dx,
\end{eqnarray}
where we used the second of \eqref{eq:inequalities}.

In the following, if $\beta\geq 1$, we will  use the following inequality
\begin{equation}\label{desi1}
2|\tilde{w}_\tau|^{\beta -1}\leq\dfrac{\beta -1}{\beta}|\tilde{w}_\tau|^{\beta}+\dfrac{1}{\beta}(2\beta)^{\beta}\leq|\tilde{w}_\tau|^{\beta}+(2\beta)^{\beta},
\end{equation}
or
\begin{equation}\label{desi2}
-\beta sign(\tilde{w}_\tau)|\tilde{w}_\tau|^{\beta -1}+|\tilde{w}_\tau|^{\beta}+(2\beta)^{\beta}\geq\beta |\tilde{w}_\tau|^{\beta -1}>0.
\end{equation}
Using the  inequalities \eqref{desi2} and \eqref{eq:inequalities}, the second term of \eqref{caos} can be estimated as
\begin{eqnarray}\label{eq:volevoilcaffe}
&&\frac{1}{k}\int_{\Omega}(|\nabla v|^{p-2}\nabla v-|\nabla u|^{p-2}\nabla u,\nabla w_{\tau})\eta^2\dfrac{1}{w_{\tau}^2}(\beta sign(\tilde {w}_\tau)|\tilde{w}_\tau|^{\beta -1}-|\tilde{w}_\tau|^{\beta}-(2\beta)^{\beta})dx\\ \nonumber
&\leq&\int_{\Omega}(|\nabla v|+|\nabla u|)^{p-2}|\nabla w_{\tau}|^2\eta^2\dfrac{1}{w_{\tau}^2}(\beta sign(\tilde {w}_\tau)|\tilde{w}_\tau|^{\beta -1}-|\tilde{w}_\tau|^{\beta}-(2\beta)^{\beta})dx\\ \nonumber
&\leq&\int_{\Omega}(|\nabla v|+|\nabla u|)^{p-2}|\nabla w_{\tau}|^2\eta^2\dfrac{1}{w_{\tau}^2}(-\beta |\tilde{w}_\tau|^{\beta -1})dx=\int_{\Omega}\rho|\nabla \tilde w_{\tau}|^2\eta^2(-\beta |\tilde{w}_\tau|^{\beta -1})dx.
\end{eqnarray}
For the third  term of \eqref{caos} we get
\begin{eqnarray}\label{caosantesde3}\\\nonumber
&&\frac{1}{k}\int_{\Omega}(a(x,v)|\nabla v|^q-a(x,u)|\nabla u|^q)\eta^2\dfrac{1}{w_{\tau}}(|\tilde{w}_\tau|^{\beta}+(2\beta)^{\beta}) dx\\ \nonumber
&=&\frac{1}{k}\int_{\Omega}(a(x,v)|\nabla v|^q-a(x,v)|\nabla u|^q+a(x,v)|\nabla u|^q-a(x,u)|\nabla u|^q)\eta^2\dfrac{1}{w_{\tau}}(|\tilde{w}_\tau|^{\beta}+(2\beta)^{\beta}) dx\\ \nonumber
&=&\frac{1}{k}\int_{\Omega}a(x,v)(|\nabla v|^q-|\nabla u|^q)\eta^2\dfrac{1}{w_{\tau}}(|\tilde{w}_\tau|^{\beta}+(2\beta)^{\beta}) dx\\\nonumber
&+&\frac{1}{k}\int_{\Omega}(a(x,v)-a(x,u))|\nabla u|^q\eta^2\dfrac{1}{w_{\tau}}(|\tilde{w}_\tau|^{\beta}+(2\beta)^{\beta}) dx.
\end{eqnarray}
Taking into account the hypotheses $({hp^*})$,  that  $\nabla u, \nabla v$ are  bounded in $\Omega$ and  that $\tau$ is a positive constant, for the two  terms on the right-hand side of
\eqref{caosantesde3} we have
\begin{eqnarray}\label{eq:laprimaaaaaaa}
&&\frac{1}{k}\left|\int_{\Omega}a(x,v)(|\nabla v|^q-|\nabla u|^q)\eta^2\dfrac{1}{w_{\tau}}(|\tilde{w}_\tau|^{\beta}+(2\beta)^{\beta}) dx\right|\\ \nonumber
&\leq&\int_{\Omega}|a(x,v)|q(|\nabla v|+|\nabla u|)^{q-1}|\nabla w_{\tau}|\eta^2\dfrac{1}{w_{\tau}}(|\tilde{w}_\tau|^{\beta}+(2\beta)^{\beta})dx\\ \nonumber
&\leq&C(q,\|v\|_{L^{\infty}(\Omega)},\|\nabla u\|_{L^{\infty}(\Omega)})\int_{\Omega}(|\nabla v|+|\nabla u|)^{q-1}|\nabla w_{\tau}|\eta^2\dfrac{1}{w_{\tau}}|\tilde{w}_\tau|^{\beta}dx\\ \nonumber
&+&C(q,\|v\|_{L^{\infty}(\Omega)},\|\nabla u\|_{L^{\infty}(\Omega)})\int_{\Omega}(|\nabla v|+|\nabla u|)^{q-1}|\nabla w_{\tau}|\eta^2\dfrac{1}{w_{\tau}}(2\beta)^{\beta}dx,
\end{eqnarray}
where we used the mean value theorem. Furthermore
\begin{eqnarray}\label{eq:mierodimenticato}
&&\frac{1}{k}\int_{\Omega}|a(x,v)-a(x,u)||\nabla u|^q\eta^2\dfrac{1}{w_{\tau}}(|\tilde{w}_\tau|^{\beta}+(2\beta)^{\beta}) dx\\ \nonumber
&\leq& C(q,L,\|\nabla u\|_{L^{\infty}(\Omega)})\int_{\Omega}\eta^2(|\tilde{w}_\tau|^{\beta}+(2\beta)^{\beta})dx.
\end{eqnarray}
Considering \eqref{eq:laprimaaaaaaa},   since $q\geq p-1$, we obtain
\begin{eqnarray}\nonumber
&&\frac{1}{k}\int_{\Omega}a(x,v)(|\nabla v|^q-|\nabla u|^q)\eta^2\dfrac{1}{w_{\tau}}(|\tilde{w}_\tau|^{\beta}+(2\beta)^{\beta}) dx\\\nonumber
&\leq&C(p,q,\|v\|_{L^{\infty}(\Omega)},\|\nabla u\|_{L^{\infty}(\Omega)},\|\nabla v\|_{L^{\infty}(\Omega)})\int_{\Omega}\rho|\nabla \tilde w_{\tau}|\eta^2|\tilde{w}_\tau|^{\beta}dx\\\nonumber
&+&C(p,q,\|v\|_{L^{\infty}(\Omega)},\|\nabla u\|_{L^{\infty}(\Omega)},\|\nabla v\|_{L^{\infty}(\Omega)})\int_{\Omega}\rho|\nabla \tilde w_{\tau}|\eta^2(2\beta)^{\beta}dx
\end{eqnarray}
and then applying (weighted) Young's inequality
\begin{eqnarray}\label{eq:thirdddddd}
&&\frac{1}{k}\int_{\Omega}a(x,v)(|\nabla v|^q-|\nabla u|^q)\eta^2\dfrac{1}{w_{\tau}}(|\tilde{w}_\tau|^{\beta}+(2\beta)^{\beta}) dx\\\nonumber
&\leq&\varepsilon C\int_{\Omega}\rho|\nabla \tilde w_{\tau}|^2\eta^2|\tilde{w}_\tau|^{\beta-1}dx+\frac C\varepsilon\int_{\Omega}\rho\eta^2|\tilde{w}_\tau|^{\beta+1}dx\\\nonumber
&+&C\int_{\Omega}\rho|\nabla \tilde w_{\tau}|^2\eta^2(2\beta)^{\beta}dx+C\int_{\Omega}\rho\eta^2(2\beta)^{\beta}dx,
\end{eqnarray}
where $\varepsilon >0$ and  $C=C(p,q,\|v\|_{L^{\infty}(\Omega)},\|\nabla u\|_{L^{\infty}(\Omega)},\|\nabla v\|_{L^{\infty}(\Omega)})$ is a positive constant.
Using \eqref{eq:mierodimenticato} and \eqref{eq:thirdddddd}, from \eqref{caosantesde3} we get
\begin{eqnarray}\label{eq:soloperto}\\\nonumber
&&\frac{1}{k}\left|\int_{\Omega}(a(x,v)|\nabla v|^q-a(x,u)|\nabla u|^q)\eta^2\dfrac{1}{w_{\tau}}(|\tilde{w}_\tau|^{\beta}+(2\beta)^{\beta}) dx\right|\\\nonumber
&\leq&C(q,L,\|\nabla u\|_{L^{\infty}(\Omega)})\int_{\Omega}\eta^2(|\tilde{w}_\tau|^{\beta}+(2\beta)^{\beta})\\\nonumber
&+&C\left(\varepsilon\int_{\Omega}\rho|\nabla \tilde w_{\tau}|^2\eta^2|\tilde{w}_\tau|^{\beta-1}dx+\frac 1\varepsilon\int_{\Omega}\rho\eta^2|\tilde{w}_\tau|^{\beta+1}dx
\int_{\Omega}\rho|\nabla \tilde w_{\tau}|^2\eta^2(2\beta)^{\beta}dx+\int_{\Omega}\rho\eta^2(2\beta)^{\beta}dx\right),
\end{eqnarray}
where $C=C(p,q,L,\|v\|_{L^{\infty}(\Omega)},\|\nabla u\|_{L^{\infty}(\Omega)},\|\nabla v\|_{L^{\infty}(\Omega)})$ is a positive constant.

Finally, using hypothesis $(hf^*)$ and the fact that $\tau$ is a positive constant,  we have
\begin{equation}\label{eq:fffffffff}
\frac{1}{k}\int_{\Omega}|f(x,v)-f(x,u)|\eta^2\dfrac{1}{w_{\tau}}(|\tilde{w}|^{\beta}+(2\beta)^{\beta}) dx\leq L\int_{\Omega}\eta^2(|\tilde{w}|^{\beta}+(2\beta)^{\beta}) dx.
\end{equation}
Collecting estimations \eqref{eq:dhaslfhgwrilaioi}, \eqref{eq:volevoilcaffe}, \eqref{eq:soloperto} and \eqref{eq:fffffffff}, from \eqref{caos} we obtain
\begin{eqnarray}\label{eq:azzzzzzzzzputtt}\\\nonumber
&&\beta\int_{\Omega}\rho|\nabla \tilde w_{\tau}|^2\eta^2|\tilde{w}_\tau|^{\beta -1}dx\\\nonumber
&\leq&C \int_{\Omega}\rho\dfrac{\eta}{w_{\tau}}(|\tilde{w}_\tau|^{\beta}+(2\beta)^{\beta})|\nabla w_\tau||\nabla\eta|  dx\\\nonumber
&+&C\left(\varepsilon\int_{\Omega}\rho|\nabla \tilde w_{\tau}|^2\eta^2|\tilde{w}_\tau|^{\beta-1}dx+\frac 1\varepsilon\int_{\Omega}\rho\eta^2|\tilde{w}_\tau|^{\beta+1}dx+\int_{\Omega}\rho|\nabla \tilde w_{\tau}|^2\eta^2(2\beta)^{\beta}dx+\int_{\Omega}\rho\eta^2(2\beta)^{\beta}dx\right)
\\\nonumber
&+&C\int_{\Omega}\eta^2(|\tilde{w}|^{\beta}+(2\beta)^{\beta}) dx\end{eqnarray}
where $C=C(p,q,L,\|v\|_{L^{\infty}(\Omega)},\|\nabla u\|_{L^{\infty}(\Omega)},\|\nabla v\|_{L^{\infty}(\Omega)})$ is a positive constant.
Now we use Young's inequality in the second term of \eqref{eq:azzzzzzzzzputtt} and  we get
\begin{eqnarray}\label{eq:ncularotta}\\\nonumber
&&\int_{\Omega}\rho\dfrac{\eta}{w_{\tau}}(|\tilde{w}_\tau|^{\beta}+(2\beta)^{\beta})|\nabla w_\tau||\nabla\eta|  dx=\int_{\Omega}\rho{\eta}|\tilde{w}_\tau|^{\beta}|\nabla \tilde w_\tau||\nabla\eta|  dx+\int_{\Omega}\rho{\eta}(2\beta)^{\beta}|\nabla \tilde w_\tau||\nabla\eta| dx\\\nonumber
&\leq&\frac\varepsilon2\int_{\Omega}\rho\eta^2|\tilde{w}_\tau|^{\beta-1}|\nabla \tilde w_\tau|^2 dx+\frac {1}{2\varepsilon} \int_{\Omega}\rho|\tilde{w}_\tau|^{\beta+1}|\nabla\eta|^2  dx\\\nonumber
&+&\frac 12\int_{\Omega}\rho(2\beta)^{\beta}|\nabla\eta|^2 dx+\frac 12\int_{\Omega}\rho{\eta^2}(2\beta)^{\beta}|\nabla \tilde w_\tau|^2 dx\,.
\end{eqnarray}
Using \eqref{eq:ncularotta} (recall \eqref{constante}) and then  grouping the term in \eqref{eq:azzzzzzzzzputtt}, for $\varepsilon$  small, we  obtain
\begin{eqnarray}\label{caos1}
&&\beta\int_{\Omega}\rho|\nabla \tilde w_{\tau}|^2\eta^2|\tilde{w}_\tau|^{\beta -1}dx\\\nonumber
&\leq&C\int_{\Omega}\rho(|\tilde{w}_\tau |^{\beta +1}+(2\beta)^{\beta})|\nabla\eta|^2dx+C\int_{\Omega}\eta^2(|\tilde{w}_\tau|^{\beta}+(2\beta)^{\beta}) dx\\\nonumber
&+&C\int_{\Omega}\rho\eta^2|\tilde{w}_\tau|^{\beta+1}dx+C(2\beta)^{\beta},
\end{eqnarray}
with $C=C(p,q,\delta, L,\|v\|_{L^{\infty}(\Omega)},\|\nabla u\|_{L^{\infty}(\Omega)},\|\nabla v\|_{L^{\infty}(\Omega)})$ some positive constant. For future use we  collect  that, to group the term
\begin{equation}\label{eq:menumale}
\int_{\Omega}\rho\eta^2(2\beta)^{\beta}dx
\end{equation}
as $C(2\beta)^\beta$ in \eqref{eq:azzzzzzzzzputtt}, we used that $p\geq2$. In the Case $(b)$ here below, we will also consider the quantity in \eqref{eq:menumale}  added to the term $C(2\beta)^\beta$ exploiting Theorem \ref{local1} that gives the right $L^1$-integrability of the weight $\rho$.

Since the support of $\eta$ depends on $\delta$, we can write
$$(2\beta)^\beta=C(\delta)\int_{B(x_0,5\delta)}(2\beta)^{\beta}\eta^2dx.$$
Thus, recalling that we are supposing $\beta \geq 1$, \eqref{caos1} (up to redefining the constant $C$ there) becomes
\begin{eqnarray}\label{caos1irrazionale}
&&\int_{\Omega}\rho|\nabla \tilde w_{\tau}|^2\eta^2|\tilde{w}_\tau|^{\beta -1}dx\\\nonumber
&\leq&C\left  (\int_{\Omega}\rho(|\tilde{w}_\tau |^{\beta +1}+(2\beta)^{\beta})|\nabla\eta|^2dx+\int_{\Omega}\eta^2(|\tilde{w}_\tau|^{\beta}+(2\beta)^{\beta}) dx+\int_{\Omega}\rho\eta^2|\tilde{w}_\tau|^{\beta+1}dx\right).
\end{eqnarray}
By   \eqref{desi1}, we also have
\begin{equation}\nonumber
\int_{\Omega}\eta^2(|\tilde{w}_\tau|^{\beta}+(2\beta)^{\beta})dx\leq\ C\int_{\Omega}\eta^2(|\tilde{w}_\tau|^{\beta +1}+(2\beta)^{\beta})dx\,.
\end{equation}
Therefore \eqref{caos1irrazionale} can be written as
\begin{equation}\label{barrafael}
\int_{\Omega}\rho|\nabla \tilde w_{\tau}|^2\eta^2|\tilde{w}_\tau|^{\beta -1}dx\leq C\int_{\Omega}(|\tilde{w}_\tau|^{\beta +1}+(2\beta)^{\beta})(\eta^2+\rho |\nabla\eta|^2)dx,
\end{equation}
with $C=C(p,q,\delta, L,\|v\|_{L^{\infty}(\Omega)},\|\nabla u\|_{L^{\infty}(\Omega)},\|\nabla v\|_{L^{\infty}(\Omega)})$ some positive constant.

We note that \eqref{barrafael} is similar to \eqref{luigipanza}, except for the extra term $(2\beta)^\beta$. Then we are going to apply again the iterative method  as we did above.

Following (taking in account the extra term $(2\beta)^\beta$) verbatim the technique from  equation \eqref{luigipanza} to  equation  \eqref{eq:robbennn}, recalling that  $r=\beta +1$ (here we are in  the hypothesis $\beta \geq1$),  using \eqref{barrafael} we obtain for $r>0$
 \begin{eqnarray}\label{eq:robbennn1}
\left(\int_{B(x,h')}(\tilde w_{\tau})^{{\frac{\nu(\beta +1)}{2}}}dx\right)^{\frac{1}{\chi r}}&=&\left(\int_{B(x,h')}\tilde w_{\tau}^{\chi r}dx\right)^{\frac{1}{\chi r }}
\\\nonumber
&\leq&\dfrac{(\dot C|r|)^{\frac 2r}}{(h''-h')^{\frac 2r}}\left [\left(\int_{B(x,h'')}\tilde w_{\tau}^rdx\right)^{\frac 1r}+ \ddot Cr\right],
\end{eqnarray}
where $\dot C=\dot C(p,q, \delta, L, \|v\|_{L^{\infty}(\Omega)},  \|\nabla u\|_{L^{\infty}(\Omega)},  \|\nabla v\|_{L^{\infty}(\Omega)})$ is a positive constant and  $\ddot C=\ddot C(p,\delta)$ is a  positive constant too that takes into account the extra term $(2\beta)^\beta$. Recalling   \eqref{eq:funzionalecolnumero}, from \eqref{eq:robbennn1} we infer that
\begin{equation}\label{34}
\phi(\chi r, h',\tilde{w}_\tau)\leq \dfrac{(\dot C|r|)^{\frac 2r}}{(h''-h')^{\frac 2r}}[\phi(r,h'',\tilde{w}_\tau)+\ddot {C}r].
\end{equation}
We claim that there exists a constant $C$ such that
\begin{equation}\label{35}
\phi(m,\dfrac{5\delta}{2},\tilde{w}_\tau)\leq C(\phi(\nu,5\delta,\tilde{w}_\tau)+m)\qquad\forall m\geq \nu.
\end{equation}
To prove the last inequality, we are going to choose
\begin{equation}\nonumber
h_k\equiv\dfrac{5\delta}{2}[1+\dfrac{1}{2^k}]\qquad k=0,1,\ldots
\end{equation}
and
\begin{equation}\label{eq:eqdiprimogrado}
 \chi ^k \nu=\chi r=\chi (\beta +1)
\end{equation}
in \eqref{34} (noting that there exits   $\beta >1$ such that \eqref{eq:eqdiprimogrado} holds), getting
\begin{equation}\nonumber
\phi(\chi ^k \nu,h_k,\tilde{w}_\tau)\leq \Big(\dfrac{2\dot C\chi^k}{h_{k-1}-h_k}\Big)^{\frac{1}{\chi^k}}[\phi(\chi^{k-1}\nu,h_{k-1},\tilde{w}_\tau)+\ddot{C}\chi^{k-1}\nu].
\end{equation}
Iterating we obtain the following
\begin{eqnarray}\label{eq:moserve}
\phi(\chi^k \nu,h_k,\tilde{w}_\tau)&\leq& (2\dot C)^{\sum_{k=1}^{+\infty}\frac{1}{\chi ^k}}\dfrac{\prod_{k=1}(\chi ^k)^{\frac{1}{\chi ^k}}}{\prod_{k=1}(h_{k-1}-h_k)^{\frac{1}{\chi ^k}}}\phi(\nu,h_0,\tilde{w}_\tau)\\\nonumber
&+&\sum_{h=0}^{k-1}\ddot {C}\chi^{h-k}\left[\prod_{\tau=h}^{k-1}\Big(\dfrac{2C\chi ^{\tau +1}}{h_{\tau}-h_{\tau +1}}\Big)^{\frac{1}{\chi ^{\tau +1}}}\right]\chi ^k \nu.
\end{eqnarray}
Therefore, estimating the products (using e.g. the logarithm function),
we obtain a positive constant $C$  such that
\begin{equation}\label{eq:sistemidiversi}
\phi(\chi ^k \nu,h_k,\tilde{w}_\tau)\leq C(\phi(\nu,5\delta,\tilde{w}_\tau)+\chi^{k}\nu),
\end{equation}
where $C= C(p,q, \delta, L, \|v\|_{L^{\infty}(\Omega)},  \|\nabla u\|_{L^{\infty}(\Omega)},  \|\nabla v\|_{L^{\infty}(\Omega)})$ (not depending on $k$) and where we have used also the fact that $h_0\equiv 5\delta$.

Setting now
$$k_m\equiv\inf_{h\in\mathbb{N}}\{h\,|\,\chi^h\nu\geq m\},$$ we obtain
\begin{eqnarray}\nonumber
\phi(m,\dfrac{5\delta}{2},\tilde{w}_\tau)&\leq &C_1\phi(\chi^{k_m}\nu,h_{k_m},\tilde{w}_\tau)\\\nonumber
&\overset{\text{by \eqref{eq:sistemidiversi}}}{\leq }& C (\phi(\nu,5\delta,\tilde{w}_\tau)+\chi^{k_m}\nu)\\\nonumber
&\leq&C(\phi(\nu,5\delta,\tilde{w}_\tau)+\chi m)\leq C(\phi(\nu,5\delta,\tilde{w}_\tau)+m),
\end{eqnarray}
up to redefine the constant $C$ and with $C= C(p,q, \delta,\Omega,  L, \|v\|_{L^{\infty}(\Omega)},  \|\nabla u\|_{L^{\infty}(\Omega)},  \|\nabla v\|_{L^{\infty}(\Omega)})$. This proves  \eqref{35}.

We are going to apply now \eqref{35} in order to prove that there exists $r_0>0$ for which \eqref{asterisco} holds. For $r_0>0$ given, taking into account the power series expansion of $e^{r_0|\tilde{w}_\tau|}$, we obtain
\begin{eqnarray}\nonumber
\int_{B(x,\frac{5\delta}{2})}e^{r_0|\tilde{w}_\tau|}dx&\leq&\sum_{k=0}^{+\infty}\int_{B(x,\frac{5\delta}{2})}\dfrac{(r_0|\tilde{w}_\tau|)^k}{k!}dx\leq\sum_{k=0}^{+\infty}\dfrac{(r_0\phi(k,\frac{5\delta}{2},\tilde{w}_\tau))^k}{k!}\\ \nonumber
&\leq&\sum_{k=0}^{+\infty}\dfrac{(Cr_0)^k(\phi(\nu,5\delta,\tilde{w}_\tau)^k+k^k)}{k!}.
\end{eqnarray}

We can prove  (using the ratio test) that, if $r_0>0$ is small enough, the last series are convergent. Therefore, it follows that

\begin{equation}\nonumber
\int_{B(x,\frac{5\delta}{2})}e^{r_0|\tilde{w}_\tau|}dx\leq C
\end{equation}
and then by monotonicity
\begin{equation}\label{eq:jsakwqDEWkabkB}
\int_{B(x,\frac{5\delta}{2})}e^{r_0\tilde{w}_\tau}dx\int_{B(x,\frac{5\delta}{2})}e^{-r_0\tilde{w}_\tau}dx\leq\left(\int_{B(x,\frac{5\delta}{2})}e^{r_0|\tilde{w}_\tau|}dx\right)^2\leq C^2.
\end{equation}
Taking now the power ${1}/{r_0}$ in the  inequality \eqref{eq:jsakwqDEWkabkB}, using \eqref{eq:funzionalecolnumero} and recalling that $\tilde{w}_\tau=log (w_{\tau})$,  we prove that \eqref{asterisco} holds   for this choice of $r_0$. Moreover, the constant $C$ in \eqref{eq:jsakwqDEWkabkB} does not depend on $\tau$.

\

\noindent {  Case $(b)$: ${(2N+2)}/{(N+2)}<p<2$.}  Arguing exactly as  in the Case $(a)$ we are able to get \eqref{luigipanza}. Using  \eqref{eq:cazzz}, we still get \eqref{beta} if $\beta \neq 1$ or \eqref{eq:bonamassaaaaa} if $\beta=1$.

Since $u\in C^1(\bar{\Omega})$, if $p<2$, the weight
$$\rho=(|\nabla u|+|\nabla v|)^{p-2}\geq \lambda >0.$$
Then as in the Case $(a)$, using the classic Sobolev's inequality (instead of Theorem \ref{bvbdvvbidvldjbvlb}), we get
\begin{equation}\nonumber
\|\eta\tilde{w}_\tau\|_{L^{2^*}(\Omega)}^2\leq C_{S}\int_{\Omega}|\nabla (\eta\tilde{w}_\tau)|^2dx\leq C(\lambda,S)\int_{\Omega}\rho|\nabla (\eta\tilde{w}_\tau)|^2dx,
\end{equation}
where $2^*$ is the classical Sobolev's exponent and therefore (see \eqref{norma})
\begin{equation}\nonumber
\|\eta\tilde{w}_\tau\|_{L^{2^*}(\Omega)}^2\leq\dot C \dfrac{1}{|\beta|}\left (1+\dfrac{1}{|\beta|}\right )r^2\int_{\Omega} \tilde{w}^2_{\tau}(\eta^2+\rho|\nabla\eta|^2)dx.
\end{equation}
Now if we suppose  $\rho\in L^t(\Omega)$, applying H\"older inequality with exponents $t$ and $t'=t/(t-1)$, we have
\begin{equation}\label{eq:cazzrrrr}
\|\eta\tilde{w}_\tau\|_{L^{2^*}(\Omega)}^2\leq Cr^2\|\rho\|_{L^t(\Omega)}\|\tilde{w}_\tau(\eta+|\nabla\eta|)\|_{L^{t^\sharp}(\Omega)}^2,
\end{equation}
where
\begin{equation}\nonumber
t^\sharp=2t'\,.
\end{equation}
We set now
$\chi'={2^*}/{t^\sharp}$ and, in order to run over again the arguments in the Case $(a)$, we only  need $\chi'>1$. This condition is obviously satisfied if
\begin{equation}\label{eq:ssssss}
t>\frac N2.
\end{equation}
By Theorem \ref{local1} (see in particular \eqref{drdrdbisssetebissete}), if $p<2$,  it follows that $\rho \in L^{\frac{p-1}{2-p}\theta}(\Omega)$, for every $0<\theta<1$. Then \eqref{eq:ssssss} holds if ${(2N+2)}/{(N+2)}<p<2$.
\end{proof}
The iteration technique is easier in the next case and it allows us to prove the following
\begin{theorem}\label{pro:p2}
Let $u,v\in C^1_{loc}(\Omega)$ and assume that either $u$ or $v$ is a weak solution to \eqref{p}, with
$ q\geq\max\,\{p-1,1\}$ and $f(x,u), a(x,u)$ satisfying $(hp^*)$.  Assume  that $\overline{B(x,6\delta)}\subset\Omega'\subset\Omega$ for some $\delta>0$ and that
\begin{equation}\label{eq:ekinotiziap}
-\Delta_pv+a(x,v)|\nabla v|^q-f(x,v)\leq -\Delta_pu+a(x,u)|\nabla u|^q-f(x,u),\quad u\leq v \quad  \text{in}\quad {B(x,5\delta)}.
\end{equation}
We distinguish the two cases:

\

\begin{itemize}
\item {Case $(a):p\geq 2$.} For all $s>1$, there exits $C>0$ such that
\begin{equation}\nonumber
\sup_{B(x,\delta)}(v-u)\leq C\|v-u\|_{L^s(B(x, 2\delta))},
\end{equation}
with $C=C(p,q, \delta, L, \|v\|_{L^{\infty}(\Omega')},  \|\nabla u\|_{L^{\infty}(\Omega')},  \|\nabla v\|_{L^{\infty}(\Omega')})$.

\

\item  { Case $(b):{(2N+2)}/{(N+2)}<p<2$.} Define
$${\bar t^\sharp}=\frac{2(p-1)}{2p-3}.$$
Then, for every $s>\bar t^\sharp/2$, there exists $C>0$ such that
\begin{equation}\nonumber
\sup_{B(x,\delta)}(v-u)\leq C\|v-u\|_{L^s(B(x, 2\delta))},
\end{equation}
with $C=C(p,q, \delta, L, \|v\|_{L^{\infty}(\Omega')},  \|\nabla u\|_{L^{\infty}(\Omega')},  \|\nabla v\|_{L^{\infty}(\Omega')})$.
\end{itemize}
\end{theorem}
\begin{proof}
We are going to use the same technique in the proof of Theorem \ref{t1} and then we will omit some details. As above we relabel the sub-domain $\Omega'$ by $\Omega$.
>From \eqref{eq:ekinotiziap} it follows that
\begin{equation}\label{difp}
-\Delta_p u+a(x,u)|\nabla u|^q-f(x,u)\geq-\Delta_p v+a(x,v)|\nabla v|^q-f(x,v)\,.
\end{equation}
In this case, given $w:=v-u$, let us define the function $\phi=\eta^2w^{\beta}$ with $\beta>0$ and $\eta\in C_0^1(B(x,5\delta))$.
For  $p\geq 2$, using $\phi$ as test function in \eqref{difp} (and repeating the same calculations of the proof to Theorem \ref{t1}),  we get (see \eqref{norma})
\begin{eqnarray}\label{normap}
\|\eta\tilde{w}\|^2_{L^{\nu}(\Omega)}&\leq&  \dot C \dfrac{1}{|\beta|}\left (1+\dfrac{1}{|\beta|}\right )r^2\int_{\Omega} \tilde{w}^2(\eta^2+\rho|\nabla\eta|^2)dx\\\nonumber
&\leq&  \dot C \dfrac{1}{|\beta|}\left (1+\dfrac{1}{|\beta|}\right )r^2\|\tilde{w}(\eta+|\nabla\eta|)\|^2_{L^2(\Omega)},
\end{eqnarray}
with $r, \beta, \dot C$ as  in \eqref{norma}. Since now $\beta >0$, it follows that $r>1$. Then, for $r>0$ (see \eqref{14} and \eqref{34} with $\ddot C=0$) we obtain
\begin{equation}\label{14p}
\phi(\chi r,h',w)\leq \dfrac{(\dot C|r|)^{\frac 2r}}{(h''-h')^{\frac 2r}} \phi(r,h'',w).
\end{equation}
Hence, taking $s>1$ and setting
$\chi ^k s=\chi r$, iterating as in \eqref{eq:moserve} and \eqref{eq:sistemidiversi},  we get
\begin{equation}\nonumber
\phi(\chi ^k s,h_k,{w})\leq C \phi(s,2\delta,{w})
\end{equation}
with $h_k$ given by
$$h_k=\delta\left (1+\left (\dfrac 12\right)^k\right ).$$
Letting $k$ tending to infinity  we have
$$\sup_{B(x,\delta)}(v-u)\leq C\|v-u\|_{L^s(B(x, 2\delta))}, $$
with $C=C(p,q, \delta, L, \|v\|_{L^{\infty}(\Omega)},  \|\nabla u\|_{L^{\infty}(\Omega)},  \|\nabla v\|_{L^{\infty}(\Omega)})$ a positive constant.

For $(2N+2)/(N+2)<p<2$ arguing as in the proof of Theorem \ref{t1}, using $\phi$  as test function in \eqref{difp}, we get (see \eqref{eq:cazzrrrr})
\begin{equation}\label{eq:cazzrrrrp}
\|\eta\tilde{w}_\tau\|_{L^{2^*}(\Omega)}^2\leq Cr^2\|\rho\|_{L^t(\Omega)}\|\tilde{w}_\tau(\eta+|\nabla\eta|)\|_{L^{t^\sharp}(\Omega)}^2,
\end{equation}
with $t^\sharp=2t'$. Iterating  \eqref{eq:cazzrrrrp} (and repeating the same type of arguments as above) we reach the desired  conclusion: given
$\bar t^\sharp={2(p-1)}/{(2p-3)}$, for any $s>\bar t^\sharp/2$ we have
$$\sup_{B(x,\delta)}(v-u)\leq C\|v-u\|_{L^s(B(x, 2\delta))}, $$
with $C=C(p,q, \delta, L, \|v\|_{L^{\infty}(\Omega)},  \|\nabla u\|_{L^{\infty}(\Omega)},  \|\nabla v\|_{L^{\infty}(\Omega)})$ a positive constant.
\end{proof}
Now it is easy to deduce the
\begin{proof}[Proof of Theorem \ref{Harna}]
The proof readily follows  applying both Theorem \ref{t1} and Theorem~\ref{pro:p2}.
\end{proof}
\noindent As a corollary of the Harnack comparison inequality we have the
\begin{proof}[Proof of Theorem \ref{thm:strongggggggg}]
The proof is standard, but we give the details for the reader's convenience. Let us set $w:=v-u$. Define the set
$$U_w=\{x\in \Omega\,|\, w(x)=0\}.$$
By the continuity of $u$ and $v$ it follows that $U_w$ is a closed set in $\Omega$. On the other hand by Theorem \ref{t1} we have that $U_w$ is also open. Then the thesis follows.
\end{proof}


\begin{thebibliography}{10}

\bibitem{lucio} L. Damascelli,
\newblock {Comparison theorems for some quasilinear degenerate elliptic operators and applications to symmetry and monotonicity results}.
\newblock {\em  Ann. Inst. H. Poincar\'{e} Anal. Non Lin\'eaire}, 15(4), pp. 493--516, 1998.

\bibitem{DS} L. Damascelli, B. Sciunzi,
\newblock Regularity, monotonicity and symmetry of positive solutions of $m$-Laplace equations.
\newblock {\em J. Differential Equations}, 206 (2), pp. 483--515, 2004.

\bibitem{DinoLu} L. Damascelli, B. Sciunzi,
\newblock Harnack inequalities, maximum and comparison principles, and regularity of positives solutions of $m$-Laplace equations.
\newblock {\em Calc. Var. Partial Differential Equations}, 25 (2), pp. 139--159, 2006.

\bibitem{Di} E.~Di Benedetto,
\newblock $C^{1+\alpha}$ local regularity of weak solutions of degenerate elliptic equations.
\newblock {\em Nonlinear Anal.}, 7(8), pp. 827--850, 1983.

\bibitem{FMRS} A.~Farina, L.~Montoro, G. Riey and B.~Sciunzi,
{\newblock  Monotonicity of solutions to quasilinear problems with a first-order
term in half-spaces.}
\newblock {\em Ann. Inst. H. Poincar\'e Anal. Non Lin\'eaire}, 32(1), pp. 1--22, 2015.




\bibitem {FMS3} A.~Farina, L.~Montoro, B.~Sciunzi,
{\newblock  Monotonicity of solutions of quasilinear degenerate elliptic equations in half-spaces.}
\newblock {\em Math. Ann.}, 357(3), pp. 855--893, 2013.


\bibitem{GiTru} D. Gilbarg, N.S. Trudinger,
\newblock Elliptic partial differential equations of second order.
\newblock {\em Reprint of the 1998 Edition}, Springer.


\bibitem{HKM} J.~Heinonen, T.~Kilpel\"{a}inen, and O.~Martio,
\newblock Nonlinear Potential Theory of Degenerate Elliptic Equations.
\newblock  {\em Oxford Mathematical Monographs}, Clarendon Press, Oxford, 1993.

\bibitem{LPR} T. Leonori, A. Porretta, G. Riey,
\newblock Comparison principles for $p$-Laplace equations with lower order terms, {\em preprint}.

\bibitem{Li} G.M.~Lieberman,
\newblock Boundary regularity for solutions of degenerate elliptic equations.
\newblock {\em Nonlinear Anal.}, 12(11), pp. 1203--1219, 1988.

\bibitem{MMPS} S. Merch\'an, L. Montoro, I. Peral, B. Sciunzi,
\newblock Existence and qualitative properties of solutions to  a quasilinear elliptic equation involving the Hardy-Leray potential.
\newblock {\em Ann. Inst. H. Poincar\'e Anal. Non Lin\'eaire}, 31(1), pp. 1--22, 2014.

\bibitem{surveymin} G. Mingione,
\newblock Regularity of minima: an invitation to the dark side of the calculus of variations. {\em Appl. Math.}, 51(4), pp. 355--426, 2006.




\bibitem{ZIGMUND} G. Mingione,
\newblock {The Calder\'{o}n-Zygmund theory for elliptic problems
with measure data}.
\newblock {\em Ann. Scuola Norm. Sup. Pisa Cl. Sci., (5).}, 6(2), pp. 195--261, 2007.


\bibitem{ming2} G. Mingione,
\newblock Gradient estimates below the duality exponent.
\newblock \emph{Math. Ann.}, 346, pp. 571--627, 2010.


\bibitem{ming3} G. Mingione,
\newblock Gradient potential estimates.
\newblock \emph{J. Eur. Math. Soc.}, 13, pp. 459--486, 2011.

\bibitem{MRS}  L.~Montoro, G. Riey and B.~Sciunzi,
{\newblock  Qualitative properties of positive solutions to systems of quasilinear elliptic equations.}
\newblock {\em Adv. Differential Equations}, 20(7-8), 717--740, 2015.

\bibitem{MSS} L. Montoro, B. Sciunzi and M. Squassina,
\newblock  Asymptotic symmetry for a class of quasi-linear parabolic problems,
\newblock  {\em Advanced Nonlinear Studies}, 10(4), pp. 789--818, 2010.


\bibitem{MO} J.K. Moser,
\newblock  On Harnack's theorem for elliptic differential elliptic equations,
\newblock  {\em Comm. on Pure and Applied Math.}, 14, pp. 577--591, 1961.

\bibitem{PSB} P.~Pucci, J.~Serrin,
\newblock \emph{The maximum principle}.
\newblock Birkhauser, Boston (2007).

\bibitem{SciDi1}{B. Sciunzi},
\newblock {Regularity and comparison principles for {$p$}-{L}aplace equations with vanishing source term}.
\newblock {\em Commun. Contemp. Math.},
 16(6), 1450013, 20 pp, 2014.

\bibitem{SciDi2}{B. Sciunzi},
\newblock {Some results on the qualitative properties of positive solutions of quasilinear elliptic equations.}
\newblock {\em NoDEA Nonlinear Differential Equations Appl.}, 14(3-4), pp. 315--334, 2007.

\bibitem{tex} E. V. Teixeira,
\newblock  Regularity for quasilinear equations on degenerate singular sets.
\newblock {\em Math. Ann.} 358(1-2), pp. 241--256, 2014.


\bibitem{tex2} E. V. Teixeira,
\newblock  Sharp regularity for general Poisson equations with borderline sources. {\em J. Math. Pures Appl.}, (9). 99(2),  pp. 150--164, 2013.

\bibitem{T} P.~Tolksdorf,
\newblock Regularity for a more general class of quasilinear elliptic equations.
\newblock {\em J. Differential Equations}, 51(1), pp. 126--150, 1984.


\bibitem{Tru} N.S. Trudinger,
\newblock  Linear elliptic operators with measurable coefficients,
\newblock  {\em Ann. Scuola Norm. Sup. Pisa.}, 27(3), pp. 265--308, 1973.

\end{thebibliography}
\end{document}